\documentclass[12pt]{amsart}
\usepackage{latexsym}
\usepackage{amssymb}
\usepackage{amsmath}

\newtheorem{Thm}[subsection]{Theorem}
\newtheorem{Lemma}[subsection]{Lemma}
\newtheorem{Rem}[subsection]{Remark}

\newtheorem{Cor}[subsection]{Corollary}
\numberwithin{equation}{section}

\newcommand{\ben}{\begin{enumerate}}
\newcommand{\een}{\end{enumerate}}

\newcommand{\bec}{\begin{center}}
\newcommand{\eec}{\end{center}}

\newcommand{\beq}{\begin{equation}}
\newcommand{\eeq}{\end{equation}}

\newcommand{\bdm}{\begin{displaymath}}
\newcommand{\edm}{\end{displaymath}}

\newcommand{\R}{\mathbb{R}}

\newcommand{\N}{\mathbb{N}}

\newenvironment{proofof}[1]{{\sc Proof of #1}}{\quad\lower0.05cm\hbox{$\square$}\medskip}

\newcommand{\bgp}{\bigskip}

\title[Kernel Estimates]
{Global Estimates for Kernels of Neumann Series and Green's Functions}

\author{Michael Frazier}
\address{Mathematics Department, University of Tennessee, Knoxville, \newline\noindent 
Tennessee 37922} \email{frazier@math.utk.edu}

\author{Fedor Nazarov}
\address{Mathematics Department, Kent State University, Kent, Ohio
44242}\email{nazarov@math.kent.edu}

\author{Igor E. Verbitsky}
\address{Department of Mathematics, University of Missouri, Columbia, Missouri 65211}
\email{verbitskyi@missouri.edu}

\thanks{The second author is supported in part
by NSF grant DMS-1249196.  The third author is supported in part by
NSF grant DMS-1161622.}



\begin{document}

\begin{abstract}
We obtain global pointwise estimates for kernels of the resolvents $(I-T)^{-1}$ 
of integral operators 
\[Tf(x) = \int_{\Omega} K(x, y) f(y) d \omega(y)\]  
on $L^2(\Omega, \omega)$ under the assumptions that $||T||_{L^2(\omega) \rightarrow L^2 (\omega)} <1$ and $d(x,y)=1/K(x,y)$ is a quasi-metric.  Let $K_1=K$ and $K_j(x,y) = \int_{\Omega} K_{j-1} (x,z) K(z,y) \, d \omega (z)$ for $j \geq 1$. Then
\[  K(x,y) e^{c K_2 (x,y)/K(x,y)}  \leq \sum_{j=1}^{\infty} K_j(x,y) \leq   K(x,y) e^{C K_2 (x,y)/K(x,y)} , \]
for some constants $c,C>0$.

Our estimates yield matching bilateral bounds for Green's functions of the 
fractional Schr\"{o}dinger operators $(-\triangle)^{\alpha/2}-q$ 
with arbitrary nonnegative potentials $q$ on $\R^n$ for 
$0<\alpha<n$, or on a bounded non-tangentially accessible domain $\Omega$ for $0<\alpha \le 2$. 
In probabilistic language, these results can be reformulated as explicit bilateral bounds for the conditional gauge 
associated with Brownian motion or $\alpha$-stable L\'evy processes. 
 
\end{abstract}

\maketitle \vfill

\eject

\tableofcontents

\section{Introduction}

This paper is dedicated to bounds for kernels of resolvents
\newline$(I-T)^{-1} $ of integral operators 
\begin{equation}
 Tf(x) = \int_{\Omega} K(x,y) f(y) \, d \omega (y) \label{defT}
\end{equation}
and their applications to estimates for Green's functions of Schr\"{o}dinger operators and related quantities.  Throughout, $\omega$ is a positive measure on $\Omega$. 

We consider the formal Neumann series
\[ (I-T)^{-1} = I + \sum_{j=1}^{\infty} T^j \]
and the associated kernels $K_1=K$ and
\begin{equation}
 K_j(x,y) = \int_{\Omega} K_{j-1} (x,z) K(z,y) \, d \omega (z) \label{defkj}
\end{equation}
for $j \geq 2$, of the operators $T^j$.  Define the formal Green's
function $H: \Omega \times \Omega \rightarrow (0, +\infty]$ by
\[ H (x,y) = \sum_{j=1}^{\infty} K_j (x,y) . \]
Let $\Vert T \Vert =\Vert T \Vert_{L^2(\omega) \rightarrow L^2
(\omega)}$ denote the operator norm of $T$ on $L^2 (\omega)$.

We will consider the class of \textit{quasi-metric} kernels, which have been considered previously in several papers, for example \cite{KV} and \cite{H}.
A quasi-metric kernel $K$ on a measure space $(\Omega, \omega)$ is a measurable function from $\Omega \times \Omega$ into $(0, \infty]$ such that

(i) $K$ is symmetric: $K(x,y) = K(y,x)$ for all $x,y \in \Omega$, 

and 

(ii) $d = 1/K$ satisfies the quasi-triangle inequality 
\begin{equation}\label{quasimetricdef} d(x,y) \leq \kappa
(d(x,z) + d(z,y) )
\end{equation}
for all $x,y,z \in \Omega$, for some $\kappa > 0$, called the quasi-metric constant for $K$. 

Our main theorem is the following.

\begin{Thm}\label{keythm} Let $(\Omega, \omega)$ be a
$\sigma$-finite measure space.  Let $K$ be a quasi-metric kernel on $\Omega$.
Suppose $\Vert T \Vert <1$. Then there exists $c = c(\kappa) >0$ and $C= C(\kappa, \Vert T \Vert)>0$ such that
\begin{equation}
   K(x,y) e^{c K_2 (x,y)/K(x,y)}  \leq H(x,y) \leq   K(x,y) e^{C K_2 (x,y)/K(x,y)}. \label{mainineq}
\end{equation}
\end{Thm}

It is well-known (see Lemma \ref{normTbig}) that if $\Vert T \Vert >1$, then 
$H(x,y) = +\infty$ for all $x$ and $y$.  In the critical case $\Vert T \Vert =1$, the lower
bound still holds, but there are examples where $H$ is finite a.e. and also examples where $H = + \infty$ a.e., 
although $K_2$ is finite a.e. 

Kernels of the form $K(x,y) = \sum_{Q} c_Q \chi_Q (x) \chi_Q(y)$, where the sum is over all dyadic cubes in $\R^n$, were considered in \cite{FV1}, in connection with a discrete model of the Schr\"{o}dinger equation (see Remark \ref{RemarkFV1}).  Such kernels are quasi-metric with quasi-metric constant $1$.  Estimates of the form of inequality (\ref{mainineq}) were obtained in \cite{FV1}, under a Carleson condition on the sequence of scalars $\{ c_Q \}$. (A sharp constant in the Carleson condition is established below; see Remark \ref{RemarkFV1}.)  

In \cite{FV}, estimate (\ref{mainineq}) and  (\ref{lowschrGest}), (\ref{upschrGest}) below were obtained under stronger
assumptions.

Estimate (\ref{mainineq}) immediately extends (see Corollary \ref{modifiableest}) to the more general class of {\it quasi-metrically modifiable} kernels. A map $K: \Omega \times \Omega \rightarrow (0, \infty]$ is quasi-metrically modifiable with constant $\kappa$ if there exists a measurable function $m: \Omega \rightarrow (0, \infty) $ such that $\widetilde{K}(x,y)=K(x,y)/(m(x) m(y))$ is a quasi-metric kernel on $\Omega$ with quasi-metric constant $\kappa$.  We call $m$ a  modifier of $K$.

Our main application is to the
fractional Schr\"{o}dinger operator
\[ \mathcal{L}_{\alpha} = (-\triangle )^{\alpha/2} - q \]
with nonnegative potential $q \in L^1_{loc}(\Omega)$ in some (possibly unbounded) domain
$\Omega \subseteq \R^n$.  Let $G(x,y)=G^{(\alpha)}(x,y)$ be the Green's kernel associated with the fractional Laplacian $(-\Delta)^{\alpha/2} $ on $\Omega$  (see \cite{L}, \cite{BBK}, \cite{H} for references and definitions).  
We note that $G(x,y)$ is non-negative and symmetric on $\R^n\times \R^n$, and $G(x, y) = 0$ if $x \in ({\overline{\Omega}})^c$, $y \in \R^n$. For regular domains $\Omega$, this is true if $x \in \Omega^c$. For the sake of simplicity we will assume throughout 
the paper that domains 
$\Omega$ are open and connected, so that $G(x, y)>0$ in $\Omega \times \Omega$, although most estimates remain true without the connectedness assumption. 

By $Gf$ we denote the 
corresponding Green potential operator, that is, 
\[ Gf(x) = \int_{\Omega} G(x,y) f(y) \, dy , \quad x \in \Omega. \] 
For appropriate $f$ and $\Omega$, we have $Gf = 0$ in $\Omega^c$ and $(-\triangle )^{\alpha/2} Gf =f$ in $\Omega$.   
More generally, 
\[ G \mu(x) = \int_{\Omega} G(x,y)  \, d\mu(y) , \quad x \in \Omega, \] 
where $\mu$ is a Borel measure on $\Omega$. 

Let $q$ be a non-negative, locally integrable function on $\Omega$. Let 
\[ d \omega(x) = q(x) dx . \]
Let $G_1 =G$ and define $G_j$ inductively for $j \geq 2$ by
\[ G_j (x,y) = \int_{\Omega} G_{j-1} (x,z) G(z,y) \, d \omega (z). \]

The minimal Green's function associated with the fractional Schr\"{o}\-dinger operator 
$ \mathcal{L}_{\alpha} = (-\triangle)^{\alpha/2}-q$  is 
\begin{equation}
 \mathcal{G} (x,y) = \sum_{j=1}^{\infty} G_j (x,y). \label{defscriptG}
\end{equation}  
The corresponding Green's operator is 
\[ \mathcal{G}f(x)  = \int_{\Omega} \mathcal{G}(x,y) f(y) \, dy .\]  

Formally, $u = \mathcal{G}f$ is the solution of the integral equation  
\[ u(x)  =   \int_{\Omega} G (x,y) \,  u(y) \, d \omega(y) + G f (x), 
\quad x \in \Omega, \quad a.e. \, \, \text{in} \, \, \Omega, \]
and hence, by applying $G$, to the Schr\"{o}dinger equation $$\mathcal{L}_{\alpha} u = (-\triangle )^{\alpha/2}u - qu = f .$$
 
Theorem \ref{mainschrothm} below, which yields estimates for $\mathcal{G}$ like those for $H$ in Theorem \ref{keythm}, is applicable in the following cases. 

(1) If $\Omega = \R^n$ and $0<\alpha<n$, then $G$ is the classical Riesz kernel $G^{(\alpha)}(x,y) = c_{n, \alpha} |x-y|^{\alpha -n}$, which is a quasi-metric kernel.

(2)  If $\Omega$ is a ball or half-space, then  for all $0<\alpha<n$,  $G=G^{(\alpha)}$ is a quasi-metrically modifiable 
kernel with modifier $m(x)=\delta(x)^{\alpha/2}$, where $\delta(x)$ is the distance from $x$ to the boundary
$\partial \Omega$. This is easy to see from the concrete form of Green's kernel in these cases.

(3) If $\Omega$ is a bounded domain with $C^{1,1}$ boundary then 
 \begin{equation}\label{green-smootha}
G^{(\alpha)}(x,y) \approx \frac{\delta(x)^{\alpha/2} \, \delta(y)^{\alpha/2}}{|x-y|^{n - \alpha} (|x-y| + \delta (x) + \delta(y))^{{\alpha/2}} },  
\end{equation}
where ``$\approx$" means that the ratio of the two sides is bounded above and below by positive constants depending only on $\Omega$, holds for $0<\alpha \le 2$, $\alpha<n$. Hence $G=G^{(\alpha)}$ is a quasi-metrically modifiable kernel with modifier $m(x)=\delta(x)^{\alpha/2}$. 
 
(4) If $\Omega$ is a bounded domain satisfying the boundary Harnark principle and $0< \alpha \le 2$, then $G=G^{(\alpha)}$ is quasi-metrically modifiable with modifier $m(x) = \min (1, G(x, x_0))$, for $x_0 \in \Omega$, with a quasi-metric constant $\kappa$ independent of $x_0$.  In particular, this procedure is applicable when $\Omega$ is a bounded Lipschitz domain, or more generally an NTA (non-tangentially accessible) domain. In fact, for $0<\alpha<2$, it suffices to assume that $\Omega$ is merely an interior NTA domain which obeys the interior corkscrew condition. This class of $\Omega$ coincides with the class of  uniform (or $\kappa$-fat) domains. See \cite{An2}, \cite{BBK}, \cite{H}, \cite{K},  \cite{FV}, p. 118, for references and further discussion.  

For the examples just listed, the following bilateral estimate follows immediately from the extension of Theorem \ref{keythm} to the case of quasimetrically modifiable kernels.  Our upper estimate is new even in the classical case $\alpha=2$; the lower estimate is known for $0<\alpha \le 2$ in the cases (1)-(4) discussed above (see \cite{GH}, and the literature cited there).  

\begin{Thm}\label{mainschrothm}
Let $\Omega \subseteq \R^n$, $n \ge 2$.  Assume that the Green's kernel $G$ for $(- \Delta)^{\alpha/2}$ on $\Omega$ is quasi-metrically modifiable. Let $q \in L^1_{loc} (\Omega)$ be non-negative, and set $d\omega = q dx$.  Define $\mathcal{G}$ by (\ref{defscriptG}). Then there exists a positive constant $c= c ( \Omega, \alpha)$ such that
\begin{equation} \mathcal G (x,y) \geq  G (x,y)  e^{\,c \,G_2 (x,y)/G (x,y) }. \label{lowschrGest}
\end{equation}
If, in addition, $\Vert T \Vert < 1$, where $T$ is the operator \newline \noindent $Tf(x) = \int_{\Omega} G(x,y) f(y) d\omega(y)$, then there exists a positive constant $C= C(\Omega, \alpha, \Vert T \Vert)$ such that
\begin{equation}
\mathcal G (x,y) \leq  G (x,y)  e^{\,C \,G_2 (x,y)/G (x,y) }.  \label{upschrGest}
\end{equation}
\end{Thm}

When $\alpha =2$, there is a precise probabilistic
formula 
\begin{equation}\label{prob}
 \mathcal{G} (x,y) / G(x,y) = {E}_{x,y} \left [ e^{\int_0^{\zeta} q(X_t) \, dt} \right ], 
 \end{equation} 
where $X_t$ is the Brownian path, with properly
rescaled time, starting at $x$, and ${E}_{x,y} $ is the conditional
expectation conditioned on the event that $X_t$ hits $y$ before
exiting $\Omega$, and $\zeta$ is the time when $X_t$ first
hits $y$.  The expression ${E}_{x,y} \left [e^{\int_0^{\zeta} q(X_t) \, dt}\right ]$  is
called the conditional gauge, or the Feynman-Kac functional 
of the conditioned process (see \cite{AS}, \cite{CZ}). Recently, similar formulas have been established  in the case $0< \alpha <2$ for the conditional gauge
associated with an $\alpha$-stable L\'evy process (see \cite{BBK}). However, our approach is more general and covers even some cases with $\alpha >2$ (in particular, $\Omega = \R^n$) for which there seems to be no probabilistic interpretation. 

This probabilistic approach  yields the lower bound 
(\ref{lowschrGest}) with $c=1$, by applying Jensen's inequality in (\ref{prob}).  On the other hand, the upper estimate (\ref{upschrGest}), which can be rewritten as 
\begin{equation}
{E}_{x,y} \left[e^{\int_0^{\zeta} q(X_t) \, dt}\right ] \leq  e^{C 
\,{E}_{x,y} \left[\int_0^{\zeta} q(X_t) \, dt\right ]},
\end{equation} 
seems to be new. It would be interesting to see if it has a probabilistic proof.  

Some nonlinear analogues of Theorem \ref{mainschrothm} for quasilinear equations of the $p$-Laplace
type with natural growth terms are obtained in \cite{JV1}, \cite {JV2}. However, they are
less precise and do not determine sharp constants in the conditions on $q$.

\vspace{0.1in}

In Section \ref{neumannseries}, we prove Theorem \ref{keythm}.  We discuss further results concerning integral operators with quasi-metric or quasi-metrically modifiable kernels in Section \ref{quasimetric}. 

\bgp

\section{Estimates for Kernels of Neumann Series}\label{neumannseries}

\begin{Lemma}\label{normTbig}  Let $(\Omega, \omega)$ be a $\sigma$-finite measure space, and let $K: \Omega \times \Omega \rightarrow (0, \infty] $ be a symmetric kernel on $\Omega$.  Define $T$ by (\ref{defT}). Let $K_1 =K$ and define $K_j$ by (\ref{defkj}) for $j \geq 2$. If $\Vert T \Vert > 1$, then for every 
$x, \,  y \in \Omega$, 
\[  H(x,y) = \sum_{j=1}^{\infty} K_j (x,y) = + \infty.\]
\end{Lemma}

\begin{proof}

Suppose there exist $x_0, y_0 \in \Omega$ such that $H(x_0, y_0) < \infty$.  Then 
\[  \int_{\Omega} H(x_0,z) K(z,y_0) \, d \omega (z)  = \int_{\Omega}
\sum_{j=1}^{\infty} K_j (x_0,z) K(z,y_0) \, d \omega (z) \]
\[  =  \sum_{j=1}^{\infty} K_{j+1} (x_0,y_0)
<  H(x_0,y_0) < \infty . \]
Since $K(z, y_0) >0$ for all $z \in \Omega$, we see that $H(x_0, z) < \infty$ for a.e. $z$.

Let $f(x) = H(x_0, x)$.  Then by the symmetry of $H$,
\[ Tf (x) = \int_{\Omega} K(x,z) H(z, x_0) \, d \omega (z) = \sum_{j
=1}^{\infty} \int_{\Omega} K(x,z) K_j (z, x_0) \, d \omega (z) \]
\[ = \sum_{j=1}^{\infty} K_{j+1} (x, x_0) < H(x, x_0) = H(x_0, x) = f(x).   \]
Since $0 < f(x) < \infty$ a.e., Schur's Lemma implies that $\Vert T \Vert
\leq 1$.  

\end{proof}

We turn to the proof of the lower estimate for $H$ in Theorem \ref{keythm}.  It is only meaningful for $\Vert T \Vert \leq 1$, by the previous lemma, but the proof does not involve $\Vert T \Vert$.

We say that $d$ is a quasi-metric with quasi-metric constant $\kappa$, on a nonempty set $\Omega$, if $d: \Omega \times \Omega \rightarrow
[0, \infty)$ is not identically $0$ and satisfies $d(x,y) = d(y,x)$ and the quasi-triangle inequality 
(\ref{quasimetricdef}) for all $x, y, z \in \Omega$.  

We will use the fact that $\kappa \geq 1/2$.  To see this fact, select $x \in \Omega$.  If $d(x,x) >0$, applying (\ref{quasimetricdef}) with $x=y=z$ gives $\kappa \geq 1/2$.  If $d(x,x) =0$, then there must exist $y \in \Omega$ such that $d(x,y) >0$, and applying (\ref{quasimetricdef}) with $z=x$ implies that $\kappa \geq 1$.  Note that $\kappa = 1/2$ is attained in the case where $d$ is constant.  

\begin{Lemma} (Ptolemy) \label{PtolemyI}
Let $d$ be a quasi-metric with constant $\kappa$ on a set $\Omega$.  Suppose $y_1, y_2, y_3, y_4 \in \Omega$.  Let $a = d(y_1, y_2), b= d(y_2, y_3), c= d(y_3, y_4), d = d(y_4, y_1), s = d(y_2, y_4)$, and $t= d(y_1, y_3)$. Then 
\[  st \leq 4 \kappa^2 \max(ac, bd). \] 
\end{Lemma}

\begin{proof}
Without loss of generality, assume $a = \min(a, b, c, d)$.  Then
\[ s \leq \kappa (a+d) \leq 2 \kappa d  \,\,\, \mbox{and} \,\,\,   t \leq \kappa (a+b) \leq 2 \kappa b. \]
Hence $st \leq 4 \kappa^2 bd$.
\end{proof}

\begin{Lemma}\label{keylowest}
Let $(\Omega, \omega)$ be a $\sigma$-finite measure space, and let $K$ be a quasi-metric kernel on $\Omega$.  Let $K_1 =K$ and define $K_j$ by (\ref{defkj}) for $j \geq 2$.  Then for $c=(4 \kappa^2)^{-1}$,  
\begin{equation}
 \sum_{j=1}^{\infty} K_j (x,y) \geq K(x,y) e^{c K_2(x,y)/K(x,y)} .  \label{Hlowest}
\end{equation}

\end{Lemma}

\begin{proof}  Let $d = 1/K$.  We can assume $d(x,y)>0$ for all $x, y$. To see this, for $n \in \N$, let
$K^{(n)} = \min (K, n)$.� Then $K^{(n)}$ is a quasi-metric with the same
quasi-metric constant as for $K$, corresponding to $d_n = \max (d, 1/n)$. 
Using (\ref{Hlowest}) for $K^{(n)}$ yields the result for $K$.  Note that at points where $K(x,y)=\infty$ or $K_2(x,y) = \infty$, both sides of (\ref{Hlowest}) are infinite.

Fix $(x,y) \in \Omega$.  For $z \in \Omega$, define 
\[ F(z) = \frac{d(x,z)}{d(y,z)} . \]
For $j \geq 2$, let 
\[ A_j = \{ (z_1, \dots , z_{j-1}) \in \Omega^{j-1} : F(z_1) \leq F(z_2) \leq \cdots \leq F(z_{j-1}) \} . \]
For $z= (z_1, \dots , z_{j-1}) \in A_j$, we have 
\[ d(x, z_{m+1} ) d(y, z_m) \geq d(x, z_m) d(y, z_{m+1}) \]
for $m=1, \dots , j-2$, and hence, by Lemma \ref{PtolemyI}, 
\[ d(z_m, z_{m+1}) d(x,y) \leq 4 \kappa^2 d(x, z_{m+1}) d(y, z_m) . \]
Therefore, letting $d \omega_{j-1}(z) = d \omega(z_1) d\omega(z_2) \cdots d \omega(z_{j-1})$, we get
\[ K_j (x,y) = \int_{\Omega^{j-1}} \frac{1}{d(x, z_1)} \frac{1}{d(z_1, z_2)} \cdots \frac{1}{d(z_{j-2}, z_{j-1})} \frac{1}{d(z_{j-1}, y)} \, d \omega_{j-1}(z)  \]
\[ \geq \left( \frac{d(x,y)}{4 \kappa^2} \right)^{j-2}  \int_{A_j} \frac{1}{d(x, z_1)} \frac{1}{d(z_1, y)} \frac{1}{d(x, z_2)} \frac{1}{d(z_2, y)} \cdots \]
\[  \cdots \frac{1}{d(x, z_{j-1})} \frac{1}{d(z_{j-1}, y)} \, d \omega_{j-1}(z).  \]
This last integral is invariant under permutations of the indices \newline \noindent $1, \dots , j-1$ in the definition of $A_j$, and hence has value at least $\frac{1}{(j-1)!}$ times the integral over all of $\Omega^{j-1}$, which splits and gives the value $K_2(x,y)^{j-1}$.  Therefore
\[ K_j (x,y) \geq c^{-1} K(x,y)  \frac{\left(c K_2 (x,y)/K(x,y) \right)^{j-1} }{(j-1)!}  , \]
with $c^{-1} = 4 \kappa^2  \geq 1.$
We sum these estimates over $j \geq 2$ and add $K(x,y)=K_1(x,y)$ to obtain
\[ \sum_{j=1}^{\infty} K_j (x,y) \geq K(x,y)  e^{c K_2 (x,y)/K(x,y)} . \]
\end{proof}

Now we turn to the upper estimate of $H$ in Theorem \ref{keythm}.  The following lemma is standard (see \cite{Hei}, Proposition 14.5), except for 
the value of the constants, which we will use. 
As indicated in the proof in \cite{Hei} (pp. 111-112), the inequality below holds with   $\beta \ge 2 \log_2 (2 \kappa)$ and $C= (2 \kappa)^2$. Notice 
that in the proof in \cite{Hei}  the quasi-ultra-metric condition $d(x,y) \le K \max [d(x,z), d(y,z)]$ 
is used in place of (\ref{quasimetricdef}), so the constant $K$ should be replaced with $2 \kappa$.)

\begin{Lemma}\label{metricequivalent}\cite{Hei} 
Let $d$ be a quasi-metric with constant $\kappa$ on set $\Omega$.  Then there exists a quasi-metric $D$ with constant 1 such that
\begin{equation}\label{bilat}
D^{\beta}  \leq d \leq CD^{\beta}
\end{equation}
for $\beta = 2 \log_2 (2 \kappa)$ and $C= (2 \kappa)^2$.  
\end{Lemma}

Note that $D$ may not be a metric because we may have $D(x,x)>0$ or $D(x,y) =0$ for $x \neq y$. The proof in \cite{Hei} can easily be adapted 
to this case. 

\begin{Lemma}(Inverse Ptolemy)\label{inversePtolemy}
Let $D$ be a quasi-metric with constant 1 on a set $\Omega$.  Let $y_1, y_2, y_3, y_4 \in \Omega$.  Let $a = D(y_1, y_2), b= D(y_2, y_3), c= D(y_3, y_4), d = D(y_4, y_1), s = D(y_2, y_4)$, and $t= D(y_1, y_3)$. Suppose the inequality 
\[  ac \geq \tau^2 bd \]
holds with some $\tau >1$.
Then 
\[  st \geq (1 - \tau^{-1})^2 ac. \]
\end{Lemma}

\begin{proof}
Without loss of generality, $a \geq c$.  Since $a^2 \geq ac \geq \tau^2 bd$, either $a \geq \tau b$ or $a \geq \tau d$.  We can assume $a \geq \tau b$.  Then 
\[ t \geq a-b \geq (1 - \tau^{-1}) a . \]
Simce $ac \geq \tau^2 bd$, either $a \geq \tau d$ or $c \geq \tau b$.  In the first case, 
\[  s \geq a-d \geq (1 - \tau^{-1})a \geq  (1 - \tau^{-1})c.  \]
In the second case, 
\[ s \geq c-b \geq   (1 - \tau^{-1})c . \]
Hence we always have the estimate
\[st \geq (1 - \tau^{-1}) a  (1 - \tau^{-1})c = (1 - \tau^{-1})^2 ac.\]

\end{proof}

\begin{Cor}\label{Ptolomyinversed}
Let $d$ be a quasi-metric with constant $\kappa$ on a set $\Omega$, and let $D$ be the quasi-metric with constant 1 determined in Lemma \ref{metricequivalent}.  Let $x,y,u, v \in \Omega$.  Suppose that for some $\tau >1$,
\[  \frac{D(x,v)}{D(y,v)} \geq \tau^2 \frac{D(x,u)}{D(y,u)}.  \]
Then
\[ d(x,y) d(u,v) \geq  (1 - \tau^{-1})^{2 \beta}(2 \kappa)^{-4}  d(x,v)d(y,u).  \]
\end{Cor}

\begin{proof}
The result can be immediately obtained by raising both sides of the estimate in Lemma \ref{inversePtolemy} to the power $\beta$ and applying the bilateral inequality (\ref{bilat}) for $d$ and $D^{\beta}$. 
\end{proof}

We now turn to the proof of the upper estimate in Theorem \ref{keythm}.
Define the quasi-metric $d = 1/K$.  Let $D$ be the quasi-metric with constant 1 determined in Lemma \ref{metricequivalent}. By considering $\min (K,n)$ as in the proof of Lemma~\ref{keylowest} and applying the monotone convergence theorem, we can assume $D(x,y)>0$ for all $x,y \in \Omega$. Fix $x, y \in \Omega$ and define the function $F$ by 
\begin{equation}\label{capitalF} 
F(z) = \frac{D(x,z)}{D(y,z)}, \quad z \in \Omega,  
\end{equation}
and $f$ by 
\begin{equation}\label{littlef}
 f(z) = \frac{1}{\sqrt{d(x,z) d(y,z)}}, \quad z \in \Omega.
 \end{equation}
Fix $\tau >1$ and $j \geq 2$.  
The heart of the proof is the following pointwise estimate.

\begin{Lemma}
Let $\tau >1$, and let $F$ and $f$ be defined by (\ref{capitalF}) and (\ref{littlef}) respectively. 
For every chain of points $z_1, z_2, \dots, z_{j-1}$ in $\Omega$, there exists a subset 
\[ M = \{ m_1 , m_2, \dots , m_{\ell} \}  \subseteq \{ 1, \dots , j-2 \} , \]
with cardinality $|M|$, with $m_1 < m_2 < \dots < m_{\ell}$ such that
\begin{equation}
F(z_{m_k}) < F(z_{m_{k+1}}) \,\, \mbox{for all} \,\,\, k = 1, 2, \dots , \ell -1,  \label{order}
\end{equation}
and
\begin{equation}
A \leq (2 \kappa)^2 C(\tau, \kappa)^{|M| } \tau^{\beta (j-2-|M|)} d(x,y)^{|M|} \, B, \label{ABineq}
\end{equation}
where 
\[  A= \frac{1}{d(x, z_1)} \frac{1}{d(z_1, z_2)} \cdots \frac{1}{d(z_{j-2}, z_{j-1})} \frac{1}{d(z_{j-1}, y)} \]
and
\begin{align*}
B &=  f(z_1)  \frac{1}{d(z_1, z_2)} \cdots  \frac{1}{d(z_{m_1 -1}, z_{m_1} )} f(z_{m_1}) \\
&\times \prod_{k=1}^{\ell -1} \left[ f(z_{m_k +1}) \frac{1}{d(z_{m_k+1}, z_{m_k +2}) } \cdots \frac{1}{d(z_{m_{k+1}- 1}, z_{m_{k+1}}) } f(z_{m_{k+1}}) \right] \\
&\times f(z_{m_{\ell}+1}) \frac{1}{d(z_{m_\ell+1}, z_{m_\ell +2}) } \cdots \frac{1}{d(z_{j-2}, z_{j-1}) } f(z_{j-1}) . 
\end{align*}

\end{Lemma}

\begin{proof}
For $m = 1, \dots , j-1$, define 
\[  \Phi (m) = \min_{k \geq m} F(z_k) . \]
Then $\Phi$ is nondecreasing.  Let 
\[ M = \{  m  \in \{1, 2, \dots , j-2  \}: \Phi (m+1) \geq \tau^2 \Phi (m) \} . \]
Notice that $F(z_m) \geq \Phi (m) $ and $F(z_m) = \Phi (m) $ for $m \in M$.  Hence $F(z_m)$ is increasing for $m \in M$, because $\Phi$ is increasing, so (\ref{order}) holds.  For $m \in M$, 
\[ F(z_{m+1}) \geq \Phi (m+1) \geq \tau^2 F(z_m) . \]
We have
\[   \frac{A}{B} = \sqrt{\frac{d(y, z_1)}{d(x, z_1)}} \left[ \prod_{m \in M} \frac{\sqrt{d(x, z_m) d(y, z_m) d(x, z_{m+1} ) d(y, z_{m+1})}}{d(z_m, z_{m+1})} \right] \sqrt{\frac{d(x, z_{j-1})}{d(y, z_{j-1})}}  .\]
The estimate $F(z_{m+1}) \geq \tau^2 F(z_m)$ means that the conditions of Corollary \ref{Ptolomyinversed} hold for the points $x, y, z_m, z_{m+1}$ for every $m \in M$.  Thus 
\[ d(x,y) d(z_m, z_{m+1}) \geq (1 - \tau^{-1})^{2 \beta} (2 \kappa)^{-4} d(x, z_{m+1}) d( y, z_m) . \]
Hence
\begin{align*}
 \frac{A}{B} &\leq  \left[ \frac{(2 \kappa)^{4}}{(1 - \tau^{-1})^{2 \beta}} \right]^{|M|}  d(x,y)^{|M|} \\
 &\times  \sqrt{\frac{d(y, z_1)}{d(x, z_1)}} \left[ \prod_{m \in M} \sqrt{ \frac{d(x, z_m) d(y, z_{m+1}) }{d(x, z_{m+1}) d(y, z_m)} }\right] \sqrt{\frac{d(x, z_{j-1})}{d(y, z_{j-1})}} .
 \end{align*}
By the equivalence of $d$ and $D^{\beta}$, we can estimate the last quantity by 
\begin{align*}
 \frac{A}{B} &\leq  \left[ (2 \kappa)^2 \right]^{|M|+1} \left[ \frac{(2 \kappa)^{4}}{(1 - \tau^{-1})^{2 \beta}} \right]^{|M|}  d(x,y)^{|M|} \\
 &\times  \left[ \frac{D(y, z_1)}{D(x, z_1)} \left( \prod_{m \in M}  \frac{D(x, z_m) D(y, z_{m+1}) }{D(x, z_{m+1}) D(y, z_m)} \right) \frac{D(x, z_{j-1})}{D(y, z_{j-1})} \right]^{\beta/2} \\
 &=  (2 \kappa)^2  \left[ \frac{(2 \kappa)^{6}}{(1 - \tau^{-1})^{2 \beta}} \right]^{|M|}  d(x,y)^{|M|}  \left[  \frac{1}{F(z_1)} \left( \prod_{m \in M} \frac{F(z_m)}{F(z_{m+1})} \right)  F(z_{j-1}) \right]^{\beta/2}
 \end{align*}
Note that $F(z_1) \geq \Phi (1), F(z_{j-1}) = \Phi(j-1)$, and recall that for every $m \in M$, we have $F(z_m) = \Phi (m)$ and $F(z_{m+1}) \geq \Phi (m+1)$.  Hence we can estimate the product 
\[ \frac{1}{F(z_1)} \left( \prod_{m \in M} \frac{F(z_m)}{F(z_{m+1})} \right)  F(z_{j-1})  \] 
by
\[ \frac{1}{\Phi (1)} \left( \prod_{m \in M} \frac{\Phi(m)}{\Phi (m+1)} \right) \Phi(j-1) = \frac{\Phi(m_1)}{\Phi(1)} \frac{\Phi (m_2)}{\Phi(m_1 +1)} \cdots \frac{\Phi (j-1)}{\Phi(m_{\ell} +1)} . \]
We now observe that the inequality $\Phi(m+1) \leq \tau^2 \Phi (m)$ holds for every $m \in [1, m_1 -1] \cup [m_1 +1, m_2 -1] \cup \dots  \cup [m_{\ell} +1,j-1]  $, i.e., for $m \notin M$, so $\Phi(m_1) \leq \tau^{2 (m_1 -1)} \Phi (1), \Phi(m_2) \leq \tau^{2 (m_2 - m_1 -1)} \Phi (m_1+ 1),$ and so on up to $ \Phi(j-1) \leq \tau^{2 (j-1- m_{\ell} -1)} \Phi( m_{\ell} +1)$.
Therefore the last product does not exceed $\tau^{2(j-2-|M|)}$.  Combining these estimates we get the conclusion of the lemma with 
$ C(\tau, \kappa) = (2 \kappa)^{6}(1 - \tau^{-1})^{-2 \beta}$. 
\end{proof}

\begin{proofof}{Theorem \ref{keythm}.}  Let $F$ and $f$ be defined by (\ref{capitalF}) and (\ref{littlef}) respectively. 
Integrating the estimate (\ref{ABineq}) with respect to $\omega_{j-1}$ and summing over all possible choices of $M$, we arrive at the inequality 
\begin{align*}
 K_j (x,y) & \leq  (2 \kappa)^2  \sum_{\ell = 0}^{j-2} C(\tau, \kappa)^{\ell } \tau^{\beta (j-1-\ell)} d(x,y)^{\ell}  \\
 & \times  \sum_{1 \leq m_1 < m_2 < \cdots < m_{\ell} \leq j-2 } I_j (m_1, m_2 , \cdots , m_{\ell}),
\end{align*}
for $j \geq 2$, where 
\begin{align*} 
& I_j (m_1, m_2 , \cdots , m_{\ell}) \\
& =  \int_{\{ F(z_{m_1}) < F(z_{m_2}) < \cdots < F(z_{m_\ell}) \} }  \\
&  \left( f T^{m_1 -1} f \right) (z_{m_1}) \left(  \prod_{k=1}^{\ell -1}  \left( f T^{m_{k+1} -m_k -1} f \right) (z_{m_{k+1}}) \right)   \left( f T^{j-1 -m_\ell -1} f \right) (z_{j-1})  \\
&  d \omega (z_{m_1}) \cdots d \omega(z_{m_\ell}) \, d\omega (z_{j-1}) .
\end{align*}
Let $\alpha  \in (1, \Vert T \Vert^{-1})$.  Define $S = \sum_{k\geq 0} \alpha ^k T^k$.  Since $f \geq 0$, we have the pointwise inequality
\[  T^k f \leq \alpha^{-k} Sf  \]
for all $k \geq 0$.  Thus 
\begin{align*}
 & I_j (m_1, \dots m_{\ell}) \leq \alpha^{-(j-2 - \ell)} \int_{\{ F(z_{m_1}) < F(z_{m_2}) < \cdots < F(z_{m_\ell}) \} } \\
& g(z_{m_1}) \cdots g(z_{m_{\ell}}) g(z_{j-1}) \ d \omega (z_{m_1}) \cdots d \omega(z_{m_\ell}) \, d\omega (z_{j-1}) \\
& \leq \frac{\alpha^{-(j-2- \ell)}}{\ell !} \Vert g \Vert_{L^1 (\omega)}^{\ell +1},
\end{align*}
where $g = fSf$.  The last inequality is true because the integral is invariant with respect to the permutations of $m_1, \dots , m_{\ell}$ in the domain of integration.

Note that 
\[ \Vert g \Vert_{L^1 (\omega)} \leq \Vert S \Vert \Vert f \Vert_{L^2 (\omega)}^2= \Vert S \Vert K_2 (x,y) . \]
Putting these estimates together, we obtain
\[ K_j (x,y) \leq (2 \kappa)^2 K_2 (x,y) \sum_{\ell = 0}^{j-2} {j-2 \choose \ell} (\tau^{\beta} \alpha^{-1})^{j-2- \ell} \frac{C(\tau, \kappa)^{\ell}}{\ell !} \left( \frac{K_2 (x,y)}{K(x,y)} \right)^{\ell} . \]

If $0 < \rho < 1$, then for each $\ell \geq 0$, 
\begin{equation}
 \sum_{j = \ell}^{\infty} {j \choose \ell} \rho^{j-\ell} = \frac{1}{(1-\rho)^{\ell +1}} , \label{diffgeom}
 \end{equation}
by differentiating $(1-\rho)^{-1} = \sum_{j=0}^{\infty} \rho^{j} $ a total of $\ell$ times. Select $\tau>1$ such that $\rho \equiv \tau^{\beta} \alpha^{-1} <1$.  Using (\ref{diffgeom}),
\[ \sum_{j=2}^{\infty} K_j (x,y) \leq \frac{(2 \kappa)^2}{1- \tau^{\beta} \alpha^{-1} } K_2 (x,y) \exp \left( \frac{C(\tau, \kappa)}{ 1 -\tau^{\beta} \alpha^{-1} } \, \frac{ K_2 (x,y)}{K(x,y)}  \right) . \]
To complete the proof of the upper bound in (\ref{mainineq}), it remains only to use the elementary inequality $1 + CV e^{CV} \leq e^{2 CV}$, valid for all $C,V>0$.
\end{proofof}

\begin{Rem}\label{RemarkFV1}{\rm 
Theorem \ref{keythm} is applicable 
to the discrete model of the Schr\"{o}dinger equation considered 
in \cite{FV1}. 
Let $\omega$ be a Borel measure on $\R^n$, and let
$\mathcal{Q}$ denote the family of dyadic cubes in $\R^n$.  For a sequence
$s=\{s_{Q} \}_{Q \in \mathcal{Q}}$ of positive scalars, we
consider an operator $T$ defined by (\ref{defT}) with kernel 
\[
 K(x,y) = \sum_{Q  \in
\mathcal{Q} } \frac{s_Q}{\omega (Q)} \chi_Q (x) \chi_Q (y), 
\]  
where the sum is taken over all dyadic cubes $Q$ such that 
$\omega(Q) \not=0$. 
This is a quasi-metric kernel with constant $\kappa=1$ (moreover, 
$d(x,y)=1/K(x,y)$ is an ultra-metric; that is, $d(x,y) \leq \max (d(x,z), d(z,y)  )$). Note that, for $g \in L^2(\omega)$,  
\[
\langle Tg, g \rangle = \sum_{Q  \in\mathcal{Q} }\frac{s_Q}{\omega (Q)} 
\left (\int_Q g d \omega\right)^{2} \le ||T|| \cdot ||g||^2_{L^2(\omega)}. 
\]
Define the discrete Carleson norm of $s=\{s_{Q} \}_{Q \in \mathcal{Q}}$ by 
\[ \Vert s \Vert_{\omega} = \sup_{Q \in \mathcal{Q}}
\omega(Q)^{-1} \sum_{P \in \mathcal{Q}: P \subseteq Q} s_P
\, \omega ( P ).\] 
Then $\Vert s \Vert_{\omega} \le ||T|| \le 4 \Vert s \Vert_{\omega}$, 
where the constant $4$ is sharp (see \cite{NTV}, Theorem 3.3). 
Consequently, by Theorem~\ref{keythm} estimate (\ref{mainineq})  holds if $\Vert s \Vert_{\omega} 
< \frac 1 4$, where the constant $\frac 1 4$ is sharp as well. Indeed,  (\ref{mainineq}) yields $||T|| \le 1$ by Schur's lemma, and so $\frac 1 4$ cannot be replaced by any larger constant in 
view of Theorem 3.3 in \cite{NTV}.

Such estimates 
of the corresponding Green's function were 
obtained earlier in \cite{FV1} by a different method,  with  
$\frac 1 {12}$ in place of $\frac 1 4$,  along with estimates of solutions to the discrete Schr\"{o}dinger equation $u=Tu +f$. }
\end{Rem}

\bgp

\section{Further Results on Quasi-metric and Quasi-metrically Modifiable Kernels}\label{quasimetric}

\bgp

Let $T$ be defined by (\ref{defT}) where $K: \Omega \times \Omega \to (0, +\infty]$ is a non-negative kernel. The minimal positive solution $u_0$ of the equation $u = Tu +1$ is obviously given by $u_0 = 1 + \sum_{j=1}^{\infty} T^j 1$.  Our next result is a bilateral pointwise estimate of $u_0$.  In the case where $K$ is the Green's function $G$ of $(-\Delta)^{\alpha/2}$, the function $u_0 = \mathcal{G} 1$ is of interest in the study of Schr\"{o}dinger equations.

\begin{Thm}\label{T1ests} Let $(\Omega, \omega)$ be a $\sigma$-finite measure space. 
Suppose that $K$ is a quasi-metric kernel with constant $\kappa$ on $(\Omega, \omega)$,   
and  $T$ is the corresponding integral operator. Then there exists $c=c (\kappa)>0$ such that the minimal positive solution $u_0$ of the equation $u = Tu + 1$
satisfies 
\begin{equation}
u_0  \geq  e^{c T1}. \label{T1lowbnd}
\end{equation}
If $\Vert T \Vert < 1$,
then there exists $C=C(\kappa, \Vert T \Vert)>0$ such that
\begin{equation}
 u_0  \leq e^{C T 1}.  \label{T1upperbnd}
\end{equation}
\end{Thm}

\begin{proof}  We first consider the case where $\Omega$ is bounded with respect to 
$d = 1/K$, that is, when $D=\sup_{x, y \in \Omega} d(x,y)<+\infty$. We will  add 
a point $z$ to $\Omega$ which is far away from all other points. That is, we choose $z \not \in\Omega$ and consider the space $\Omega^* = \Omega \,  \cup \{z\}$ with quasi-metric  $d^*$ 
defined by  $d^* (x,y) = d(x,y) $ if $x,y \in \Omega$,  $d^*(x,z)= d^*(z,x)= D$ for all $x \in \Omega$, 
and $d(z,z)=0$.  Then $d^*$ is a quasi-metric on $\Omega^*$ with quasi-metric constant $\kappa^* = \max (\kappa, 1)$.  We also extend $\omega$ to a measure $\omega^*$ on $\Omega^*$ by setting $ \omega^*\vert_{\Omega} = \omega$, and $\omega^* (\{z\}) = 0$.  

Note that the iterates $K^*_j$ of $K^* = 1/d^* $ with respect to $\omega^*$ agree with the iterates $K_j$ of $K$ with respect to $\omega$ on $\Omega \times \Omega$ since  
$\omega^* (\{ z \}) =0$, and that the norm of the integral operator with the kernel $K^*$ 
on $(\Omega^*, \omega^*)$ is the same as $||T||$. 

For all $x \in \Omega$,
\[ \frac{K_2^* (x,z)}{K^*(x,z)} = D \int_{\Omega^*} K^*(x,y) K^*(y,z) \, d \omega^* (y) =  \int_{\Omega} K(x,y) \, d \omega (y) =  T1 (x)  \]
and
\[ u_0 (x) = 1 + \sum_{j=1}^{\infty} \int_{\Omega}  K_j (x,y) \, d \omega (y) = 1 + D \sum_{j=1}^{\infty}  \int_{\Omega^*}  K_j^* (x,y) K^*(y,z) \, d\omega^* (y) \]
\[ = D K^*(x,z)  + D \sum_{j=2}^{\infty} K_{j}^* (x,z) =  D \sum_{j=1}^{\infty} K_j^* (x,z) . \]
Hence, applying the lower estimate in Theorem \ref{keythm} on the space $\Omega^*$ we get, 
for all $x \in \Omega$,
 \begin{equation}
 u_0 (x) \geq D K^*(x,z) e^{c K_2^* (x,z)/K^*(x,z)} = e^{c T1 (x)}.  \label{u0lowest}
\end{equation}
Similarly, the upper estimate in Theorem \ref{keythm} gives, for all $x \in \Omega$, 
\begin{equation}
 u_0 (x) \leq D K^*(x,z) e^{C K^*_2 (x,z)/K^*(x,z)} = e^{C T1 (x)}. \label{u0upest}
 \end{equation}
  
For $\Omega$ not bounded with respect to $d$, select $x_0 \in \Omega$ and let $\Omega_n = \{ x \in \Omega : d(x, x_0) <n \}$.  Let $\omega_n $ be the restriction of $\omega$ to $\Omega_n$, and let $d_n$ and  $K^{(n)}$ be the restrictions of $d$ and $K$ to $\Omega_n \times \Omega_n$ respectively.  Then $K^{(n)}$ is a quasi-metric kernel on $\Omega_n$.  The corresponding integral operator 
  $T_n$ defined by 
\[ T_n f(x) = \int_{\Omega_n} K^{(n)} (x,y) f(y) \, d \omega_n (y) = \int_{\Omega} K(x,y) \chi_{\Omega_n} (y) \, d \omega (y) \]
satisfies 
\[ \Vert T_n \Vert_{L^2 (\Omega_n) \rightarrow L^2 (\Omega_n) } \leq \Vert T \Vert_{L^2 (\Omega) \rightarrow L^2 (\Omega)} ,  \] 
and $T_n 1 \to T 1$ pointwise as $n \to \infty$. 

Let $T^{j}_n$ be the $j^{th}$ iterate of $T_n$ and let $K^{(n)}_j$ be the kernel of $T^{j}_n$.  Then $K^{(n)}_j (x,y) $ is non-decreasing in $n$ and converges to $K_j (x,y)$ pointwise  as $n \rightarrow \infty$, for each $j \in \N$,  by the monotone convergence theorem.  Let $u_0^{(n)} = 1 + \sum_{j=1}^{\infty} T^{j}_n1$.  By the monotone convergence theorem, $ u_0^{(n)} \to  u_0$ pointwise as $n \to \infty$. 

Applying the estimates  for the bounded space $\Omega_n$, and passing to the limit as $n \to \infty$, we see that estimates (\ref{T1lowbnd}) and (\ref{T1upperbnd}) hold in the unbounded case as well.
\end{proof}

\vspace{0.1in}

We now turn to  characterizing the kernels for which
$\sum_{j=1}^{\infty} K_j (x,y)$ is pointwise equivalent to
$K(x,y)$.

\begin{Thm}\label{charVeq}
Let $(\Omega, \omega)$ be a $\sigma$-finite measure space. Suppose
$K: \Omega \times \Omega \rightarrow (0, + \infty]$ is a 
quasi-metric kernel with constant $\kappa$, and that $K$ is not identically $\infty$.  Let $T$ be the integral operator corresponding to $K$ and let $u_0$ be the minimal positive solution of the equation $u=Tu +1$. Then the following statements are equivalent:

\vspace{0.1in}

(a)  There exists $C_1>0$ such that $\sum_{j=1}^{\infty} K_j (x,y) \leq C_1 K(x,y)$ for all $x,y \in \Omega$. 

\vspace{0.1in}

(b)  $\Vert T \Vert <1$ and $K_2 (x,y) \leq C_2 K(x,y)$ for all $x,y \in \Omega$, for some $C_2>0$ (or, equivalently, $\sup_{x \in \Omega} T1(x)< +\infty$).

\vspace{0.1in}

(c)  $\sup_{x \in \Omega} u_0(x)< +\infty$. 

\end{Thm}

\begin{proof} 
We first show that the condition $K_2 \leq C_2 K$ is equivalent to the boundedness of $T1$.  Notice that by the quasi-metric 
property of $K$,  
\[
K(x, z) K(y,z) \le \kappa K(x,y) [ K(x, z) +  K(y,z)]. 
\]
Hence, 
\begin{align*}
K_2(x,y) & =\int_\Omega K(x, z) K(y,z) d \omega (z)\\ & \le  \kappa K(x,y) \int_{\Omega} K(x, z) d \omega (z)
 + \kappa K(x,y) \int_{\Omega} K(y,z) d \omega (z)\\ &
= \kappa K(x,y) [ T1(x) + T1(y)],
\end{align*}
so that the boundedness of $T1$ implies that $K_2 \leq C_2 K$. 

Now suppose $K_2 \leq C_2 K$.  Fix $x \in \Omega$.  Suppose first that  

(1) $0< \sup_{y \in \Omega} d(x,y)=D< +\infty$. Pick any $y \in \Omega$ with $d(x,y)> \frac {D}{2}$. Then 
\[
d(y,z) \le \kappa [d(x,y)+d(x,z) ] \le 3 \kappa d(x,y)
\]
for all $z \in \Omega$, and so $K(y,z) \ge \frac {1}{3 \kappa} K(x, y)$.  It follows that 
\[
K_2(x,y) =\int_\Omega K(x, z) K(y,z) d \omega (z)\ge \frac {1}{3 \kappa} K(x, y) T1(x),  
\]
whence 
$T1(x) \le 3 \kappa C_2$. 

Now suppose that

(2) $\sup_{y \in \Omega} d(x,y)=+\infty$. Then there is a sequence $y_n \in \Omega$ such that 
$0<r_n=d(x, y_n) \to +\infty$ as $n \to \infty$. For every $z \in B(x, r_n)$, we have 
\[
d(y_n,z) \le \kappa [d(x,y_n)+d(x,z) ] \le 2 \kappa d(x, y_n).  
\]
Hence, 
\begin{align*}
K_2(x, y_n) & \ge \int_{B(x, r_n)} K(x,z) K(y_n,z) d \omega (z)\\
 & \ge \frac{1}{2 \kappa} 
K(x, y_n) \int_{B(x, r_n)} K(x,z) d \omega(z),  
\end{align*}
and consequently $\int_{B(x, r_n)} K(x,z) d \omega(z)\le 2 \kappa C_2$ for all $n \in \N$. 
Passing to the limit as $n \to \infty$, we get $T1(x) \le 2 \kappa C_2$. 

Now Theorem \ref{keythm} shows that (b) implies (a), and Theorem \ref{T1ests} shows that (b) implies (c).

If $u_0$ is bounded by $C$, then by Theorem \ref{T1ests}, $T1$ is bounded.  From $u_0 = Tu_0 +1$, we obtain
\[ Tu_0 = u_0 -1 \leq \left(1 - \frac{1}{C} \right) u_0. \] 
Hence $\Vert T \Vert \leq 1- 1/C <1$, by Schur's Lemma.  So (c) implies (b).

It remains to show that (a) implies (b).  Condition (a) trivially implies $K_2 \leq C_2 K$, and hence we have that $T1$ is bounded.  It remains to  show that (a) implies $||T||<1$.  Since the kernel of $T$ is positive,  $T$ is bounded on $L^\infty(\omega)$, and by duality on $L^1(\omega)$. Thus by interpolation $T$  is a bounded operator on $L^2(\omega)$. Comparing kernels and using (a), there exists $C$ so that for all $n$, 
\[ \Vert T^n \Vert \leq \Vert T + T^2 + \cdots + T^n \Vert \leq C . \]
Since the kernel of $T$ is symmetric, $T$ is self-adjoint, so $||T||$ coincides with the spectral radius $r(T)$ on $L^2(\omega)$.  Hence it suffices to show that $\Vert T^n \Vert <1$ for some $n$.  If not, given $n$ we can select $f$ such that $\Vert f \Vert =1$ and $\Vert T^n f \Vert > 1/2$.  We can assume $f \geq 0$. Then for all $m \leq n$
\[ \Vert T^m f \Vert \geq \frac{1}{2C} . \]
Then
\[ \Vert (T + T^2 + \cdots +T^n) f \Vert^2 \geq \sum_{j=1}^n \Vert T^j f \Vert^2 \geq \frac{n}{4C^2}, \]
since all inner products in the expansion of the left side are non-negative.  For $n$ large enough, this inequality contradicts $\Vert T + T^2 + \cdots + T^n \Vert \leq C$.

\end{proof}

{\bf Remarks.} 
1.  The condition $\sup_{x \in \Omega} T1(x)< +\infty$ can be expressed in a ``geometric form'' 
\[\sup_{x \in \Omega} \int_{0}^{+\infty} \frac{\omega(B(x,t))}{t^2}dt< + \infty,\]
where $B(x,t) = \{ y \in \Omega: d(x,y)<t\}$.  Indeed,
\[ T1(x)=\int_\Omega K(x, y) d \omega (y) = \int_\Omega \frac{d\omega(y)}{d(x,y)}\]
\[ =    \int_{ \Omega} \int_{d(x,y)}^{\infty}  \frac{dt \, d\omega(y) }{t^2}=\int_{0}^{+\infty} \frac{\omega(B(x,t))}{t^2}dt.    \]

      \vspace{0.1in}
   
   2. The condition $\sup_{x \in \Omega} T1(x)< +\infty$ can be replaced with \newline $||T1||_{L^\infty(\omega)} < + \infty$, 
   which in its turn is equivalent to $||T||_{L^1(\omega) \to L^1(\omega)} < + \infty$. 
   
   Indeed, let $ E=\{x: T1(x) \le ||T1||_{L^\infty(\omega)} \}$. Then $\omega(\Omega\setminus E)=0$, so $E$ 
   is non-empty. Fix any point $x \in \Omega$. Then the following two cases are possible: (i) $\inf_{y\in E} 
   d(x,y)=0$. In this case, for every $\epsilon>0$, there exists $y \in E$ such that 
   $d(x,y)< \epsilon$, and therefore  $d(y,z)\le \kappa \left ( d(x,z)+\epsilon \right)$, whence 
   \[
   \int_{\Omega} \frac{d \omega(z)}{d(x,z) + \epsilon} \le \kappa T1(y) \le \kappa ||T1||_{L^\infty(\omega)}.
   \]
   Passing to the limit as $\epsilon\to 0$, we obtain $ T1(x) \le \kappa ||T1||_{L^\infty(\omega)}$. 
   
   (ii) $D=\inf_{y\in E} d(x,y) >0$. Then choose any point $y \in E$ with 
   $d(x,y) \le 2D$. Note that for all $z \in E$, 
   \[
   d(y,z) \le \kappa \left ( d(x,y) + d (x,z)\right) \le \kappa \left (2D + d (x,z)\right)\le 3 \kappa d(x,z).
   \] 
   Thus in this case 
   \[
   T1(x) =\int_{E}\frac{d \omega(z)}{d(x,z)} \le 3 \kappa  \int_{E}\frac{d \omega(z)}{d(y,z)} 
   =3 \kappa T1(y)\le 3 \kappa ||T1||_{L^\infty(\omega)}.
   \]
      \vspace{0.1in}
      
      3. Let $\mathcal{B}$ denote the space of all bounded 
functions on $\Omega$ with norm $||u||_{\mathcal{B}} = \sup_{x \in \Omega} |u(x)|$. 
Suppose $T$ is an integral operator with quasi-metric kernel.  Clearly $T: \mathcal{B} \to \mathcal{B}$ is bounded if and only if $\sup_{x \in \Omega} T1(x)< +\infty$.  
Under this additional assumption, W. Hansen \cite{H}  showed that condition (a) 
of Theorem~\ref{charVeq} is equivalent 
to $r(T)_{\mathcal{B}}<1$, where $r(T)_{\mathcal{B}}$ is the spectral radius of  $T$ in $\mathcal{B}$. This result is a consequence of Theorem~\ref{charVeq} above.  Moreover, 
for operators $T$ with quasi-metric kernels which are bounded on $\mathcal{B}$, we have $r(T)_{\mathcal{B}}=||T||$. 

Indeed,  $||T||_{\mathcal{B} \to \mathcal{B}} \ge ||T||_{L^\infty (\omega) \to L^\infty (\omega)} 
=||T||_{L^1(\omega) \to L^1(\omega)}$. Using interpolation, 
and the formula $r(T)_{\mathcal{B}}=\lim_{n \to \infty} ||T^n||_{\mathcal{B} \to \mathcal{B}}^{1/n}$ we see that $r(T)_{\mathcal{B}}\ge r(T)_{L^2(\omega)}=||T||$.  By an argument similar to that used in the proof of Theorem \ref{charVeq}  (with $\mathcal{B}$, or $L^\infty(\omega)$,  in place of $L^2(\omega)$) it follows that (a) implies 
$r(T)_{\mathcal{B}}<1$. Thus, the condition $r(T)_{\mathcal{B}}<1$ is equivalent to 
$||T||<1$ for operators $T$ with quasi-metric kernels bounded on $\mathcal{B}$. To prove that $r(T)_{\mathcal{B}} = ||T||$, it remains to notice that, for $\epsilon >0$, the operator $T_\epsilon =(||T|| + \epsilon)^{-1} T$ satisfies $||T_\epsilon||<1$, and hence 
$r(T_\epsilon)_{\mathcal{B}} <1$, which yields $r(T)_{\mathcal{B}} < ||T|| + \epsilon$. Conversely, $S_\epsilon=( r(T)_{\mathcal{B}} + \epsilon)^{-1} T$ satisfies 
$r(S_\epsilon)_{\mathcal{B}} <1$ which gives $||S_\epsilon||<1$, that is,  $||T|| < r(T)_{\mathcal{B}} + \epsilon$. 
Letting $\epsilon \to 0$ yields $r(T)_{\mathcal{B}} = ||T||$. 
   \vspace{0.1in}

%
%
   
   \vspace{0.1in}
\bgp

We will now extend our results to a wider class of quasi-metrically modifiable kernels.  It turns out that in many interesting applications the quasi-metric property fails but the quasi-metric modifiability holds.  Let $K$ be a quasi-metrically modifiable kernel on a measure space $(\Omega, \omega)$, with modifier $m$, so that $\widetilde{K}(x,y)=K(x,y)/(m(x) m(y))$ is a quasi-metric kernel.

We also consider the measure $d \widetilde{\omega} = m^2 d \omega$ and the operator $\tilde{T}$ defined by 
\[ \widetilde{T} f(x) = \int_{\Omega} \widetilde{K} (x,y) f(y) \, d \tilde{\omega} (y).  \]
Various properties of a quasi-metrically modifiable kernel $K$ and the corresponding integral operator $T$ can be reduced to 
those of $\widetilde{K}$ and $\widetilde{T}$. The following properties are straightforward, and we leave their proofs to the reader. 

   \vspace{0.1in}
(a) $\widetilde{K}_j(x,y) = \frac{K_j(x,y)}{m(x) m(y)}$ for all $j \in \mathbb{N}$.    \vspace{0.1in}

(b) If $f \in L^2(\omega)$, then $\widetilde{f} = \frac{f}{m} \in L^2 (\widetilde{\omega})$. 
   \vspace{0.1in}

$($c$)$ $\widetilde{T}^j\widetilde{f}= \frac{T^jf}{m}$, for all $j \in \N$.
   \vspace{0.1in}

(d) $\widetilde{T}^j 1= \frac{T^jm}{m}$, for all $j \in \N$. 
    \vspace{0.1in}

(e) $||\widetilde{T}||_{L^2 (\widetilde{\omega})} = ||T||_{L^2 (\omega)}$.
   \vspace{0.1in}

   \vspace{0.1in}

Applying Theorem~\ref{keythm} to $\widetilde{K}$ and $\widetilde{T}$, and 
rewriting the conclusions in terms of $K$ and $T$, we deduce that Theorem~\ref{keythm} 
remains valid verbatim for quasi-metrically modifiable kernels, which we state 
as the following corollary.  

\begin{Cor}\label{modifiableest}
Let $(\Omega, \omega)$ be a $\sigma$-finite measure space, and let $K$ be a quasi-metrically modifiable kernel on $\Omega$ with constant $\kappa$.  Let $K_1 =K$ and define $K_j$ by (\ref{defkj}) for $j \geq 2$. Then there exists $c>0$, depending only on $\kappa$, such that (\ref{Hlowest}) holds.  

Define $T$ by (\ref{defT}).  If $\Vert T \Vert_{L^2 (\omega) \rightarrow L^2 (\omega)} < 1,$ then there exists $C>0$, depending only on $\kappa$ and $\Vert T \Vert$, such that 
\[ \sum_{j=1}^{\infty} K_j (x,y) \leq  K(x,y) e^{C K_2 (x,y)/K(x,y)} . \]
\end{Cor}

Theorem~\ref{T1ests} becomes a statement concerning the minimal positive solution $u_0$, defined by $u_0 = m + \sum_{j=1}^{\infty} T^j m$, of the equation $u = Tu +m$, where $m$
is the quasi-metric modifier. We obtain estimates for $u_0$ in the next corollary which is deduced by applying Theorem~\ref {T1ests} to $u_0/m = 1 + \sum_{j=1}^{\infty} \widetilde{T}^j 1$.   

\begin{Cor}\label{Cests}
Suppose $K$ is a quasi-metrically modifiable kernel with modifier
$m$ and constant $\kappa$ on $(\Omega, \omega)$.  Define $T$ by
(\ref{defT}). Then there exists $c>0$ depending only on $\kappa$ such that 
\begin{equation}
u_0  \geq  m e^{c (Tm) /m}. \label{T1lowbnd1}
\end{equation}
If $\Vert T \Vert < 1$,
then there exists $C>0$ depending only on $\kappa$ and $\Vert T \Vert$ such that
\begin{equation}
 u_0  \leq m e^{C (T m) /m}.  \label{T1upperbnd1}
\end{equation}
Moreover, $u_0 \approx m$ if and only if $||T||<1$, and $T m \le C  m$. 
\end{Cor}
\vspace{0.1in}

The next corollary is a direct analogue of Theorem~\ref{charVeq} 
for quasi-metrically modifiable kernels, proved by reducing to the quasi-metric case via (a) - (e) above.

\begin{Cor}\label{charVequivK}

Let $(\Omega, \omega)$ be a $\sigma$-finite measure space. Suppose
$K: \Omega \times \Omega \rightarrow (0, + \infty]$ is
quasi-metrically modifiable with constant $\kappa$ and modifier
$m$, and that $K$ is not identically $\infty$. Let $T$ be the integral operator corresponding to $K$ and let $u_0$ be the minimal positive solution of the equation $u=Tu +m$. Then the following statements are equivalent:

\vspace{0.1in}

(a)  There exists $C_1>0$ such that $\sum_{j=1}^{\infty} K_j (x,y) \leq C_1 K(x,y)$ for all $x,y \in \Omega$. 

\vspace{0.1in}

(b)  $\Vert T \Vert <1$ and $K_2 (x,y) \leq C_2 K(x,y)$ for all $x,y \in \Omega$, for some $C_2>0$ (or, equivalently, $\sup_{x \in \Omega} (Tm(x))/m(x) < +\infty$).

\vspace{0.1in}

(c)  $\sup_{x \in \Omega} (u_0(x)/m(x))< +\infty$. 

\end{Cor}

\vspace{0.1in}

In conclusion of this section we discuss an intrinsic characterization of 
the class of quasi-metrically modifiable kernels.  Recall that a  positive symmetric 
kernel $K$ is quasi-metrically modifiable with modifier $m>0$ if and only 
if $d(x,y) = m(x)m(y) D (x,y)$, where $ D (x,y) = 1/K(x,y)$,  
is a quasi-metric.  
Then $D$ satisfies the Ptolemy inequality:
\[ D(y_1, y_3) D(y_2, y_4) \leq 4 \kappa^2 \left( D(y_1, y_2) D(y_3, y_4) + D(y_1, y_4) D(y_2, y_3)   \right)  \]
for all $y_1, y_2, y_3, y_4 \in \Omega$.  Indeed, by Lemma \ref{PtolemyI}, we have such an inequality for $d$.  Using the relation between $d$ and $D$ and cancelling the term $m(y_1) m(y_2)m(y_3)m(y_4)$ yields the Ptolemy inequality for $D$.

On the other hand, suppose $K$ is a postive, symmetric kernel with the property that for some $w \in \Omega$, $K(x,w) < \infty$ for all $x$, such that $D=1/K$ satisfies the Ptolemy inequality 
\[ D(x,y) D(z,w) \leq C \left( D(x,z) D(y,w) + D(x,w) D(y,z)  \right) \]
for all $x,y,z \in \Omega$.  Then dividing by $D(x,w)D(y,w)D(z,w)$ yields
\[ \frac{D(x,y)}{ D(x,w)D(y,w) } \leq C \left(  \frac{D(x,z)}{ D(x,w)D(z,w) } + \frac{D(y,z)}{D(y,w) D(z,w) }  \right).  \]
Hence $m(x) = 1/D(x,w)= K(x,w)$ is a quasi-metric modifier for $K$. More generally, if $\omega(\{x: K(x,w) = \infty \}) =0$, then $\{x: K(x,w) = \infty \}$ can be deleted from $\Omega$ without significant effect, and the contrary case is somewhat degenerate. Such observations about quasi-metric modifiers were first noticed by Hansen and Netuka \cite{HN} (Proposition 8.1).

\bgp

\end{document}